\documentclass[11pt]{amsart} 
\usepackage{amssymb,amsmath,latexsym,enumerate,graphicx,bbm,mathptmx,microtype,cite,ifthen,color}
\allowdisplaybreaks

\usepackage{cmbright}

\newtheorem{theorem}{Theorem} 
\newtheorem{corollary}[theorem]{Corollary}

\newtheorem{exam}{Example}
\newenvironment{example}{\begin{exam}\sf}{\end{exam}}
\newtheorem*{rem}{Remarks}

\renewcommand\th{^{\text{th}}}

\newcommand\commentout[1]{}
\newcommand\Def[1]{{\bf #1}}

\newcommand\des{\operatorname{des}} 
\newcommand\Des{\operatorname{Des}} 
\newcommand\maj{\operatorname{maj}}

\newcommand\chap{\operatorname{cha}} 
\newcommand\ehr{\operatorname{ehr}}

\newcommand\ZZ{\mathbb{Z}}
\newcommand\QQ{\mathbb{Q}}
\newcommand\RR{\mathbb{R}}

\newcommand\poly{\mathcal{P}}
\newcommand\qoly{\mathcal{Q}}
\newcommand\kegel{\mathcal{K}}

\newcommand\be{\mathbf{e}}
\newcommand\bg{\mathbf{g}}
\newcommand\bm{\mathbf{m}}
\newcommand\bv{\mathbf{v}}
\newcommand\bw{\mathbf{w}}
\newcommand\bx{\mathbf{x}}

\newcommand\bz{\mathbf{z}}
\newcommand\bzero{\mathbf{0}}

\makeatletter 
\newtheorem*{rep@theorem}{\rep@title}\newcommand{\newreptheorem}[2]{%
\newenvironment{rep#1}[1]{%
\def\rep@title{\bf #2 \ref{##1}}%
\begin{rep@theorem}}%
{\end{rep@theorem}}}
\makeatother
\newreptheorem{theorem}{Theorem}

\newcounter{teach}
\begin{document}

\title{A closer look at Chapoton's $q$-Ehrhart polynomials}

 \author{Matthias Beck}
 \address{Department of Mathematics\\
          San Francisco State University\\
          San Francisco, CA 94132\\
          U.S.A.}
 \email{mattbeck@sfsu.edu}

 \author{Thomas Kunze}
 \address{Department of Mathematics\\
          University of California\\
          Irvine, CA 92697\\
          U.S.A.}
 \email{tkunze@uci.edu}


\begin{abstract}
If $\poly$ is a lattice polytope (i.e., $\poly$ is the convex hull of finitely many integer
points in~$\RR^d$), Ehrhart's famous theorem (1962) asserts that the integer-point counting function $|t \poly \cap \ZZ^d|$ is a polynomial in the integer variable~$t$.
Chapoton (2016) proved that, given a fixed integral form $\lambda: \ZZ^d \to \ZZ$, there exists a
polynomial $\chap_\poly^\lambda(q,x) \in \QQ(q)[x]$ such that the refined enumeration function $\sum_{ \bm \in t \poly } q^{ \lambda(\bm) }$ equals the evaluation $\chap_\poly^\lambda (q, [t]_q)$ where, as usual, $[t]_q := \frac{ q^t - 1 }{ q-1 }$; naturally, for $q=1$ we recover the Ehrhart polynomial.
Our motivating goal is to view Chapoton's work through the lens of Brion's Theorem (1988), which expresses the integer-point structure of a given polytope via that of its vertex cones.
It turns out that this viewpoint naturally yields various refinements and extensions of
Chapoton's results, including explicit formulas for $\chap_\poly^\lambda(q,x)$, its leading
coefficient, and its behavior as $t \to \infty$. We also prove an analogue of Chapoton's
structural and reciprocity theorems for rational polytopes (i.e., with vertices in~$\QQ^d$).
\end{abstract}



\date{2 May 2026}

\thanks{We are grateful to Tewodros Amdeberhan, Fr\'ed\'eric Chapoton, Serkan Ho\c sten, Leonid Monin, Dusty Ross,
Raman Sanyal, Christopher Voll, and an anonymous referee for helpful remarks and pointers to the literature.}

\maketitle


\section{Chapoton Polynomials}

Let $\poly \subset \RR^d$ be a \Def{lattice polytope}, i.e., the convex hull of finitely many points in $\ZZ^d$.
The \Def{Ehrhart polynomial} of $\poly$ is defined via $\ehr_\poly(t) := |t \poly \cap \ZZ^d|$ for $t \in \ZZ_{
>0 }$~\cite{ehrhartpolynomial}. It is an important invariant of $\poly$ and enjoys numerous applications in various mathematical fields (see, e.g.,~\cite{ccd}). 
Chapoton~\cite{chapoton} initiated the study of
\[
  \ehr_\poly^\lambda(q,t) \, := \sum_{ \bm \in t \poly \cap \mathbb{Z}^d} q^{ \lambda(\bm) }
\]
where $\lambda: \ZZ^d \to \ZZ$ is a fixed integral form. Naturally, for $q=1$ we recover the Ehrhart polynomial
of $\poly$, but the algebraic structure of $\ehr_\poly^\lambda(q,t)$ is a priori not clear.
Chapoton's main theorem~\cite{chapoton} is as follows.
We define, as usual, $[t]_q := \frac{ q^t - 1 }{ q-1 }$ and 
we call two vertices $\bv, \bw$ of $\poly$ \Def{adjacent} if they form an edge of $\poly$.
We say that $\lambda$ is \Def{generic and positive with respect to} $\poly$ if $\lambda(\bv) \ne
\lambda(\bw)$ whenever $\bv$ and $\bw$ are adjacent vertices, and $\lambda(\bv) \ge 0$ for any
vertex~$\bv$.

\begin{theorem}[Chapoton]\label{thm:chapoton}
If $\poly$ is a lattice polytope and $\lambda$ is an integral form that is generic and positive with
respect to $\poly$, then there exists a polynomial $\chap_\poly^\lambda(q,x) \in \QQ(q)[x]$ such
that \[ \chap_\poly^\lambda \left(q, [t]_q \right) \, = \, \ehr_\poly^\lambda(q,t) \] for all positive integers $t$.
The degree of $\chap_\poly^\lambda(q,x)$ is $m := \max \{ \lambda(\bv) : \, \bv \text{
\rm vertex of } \poly \}$, and all the poles of the coefficients of $\chap_\poly^\lambda(q,x)$
are roots of unity of order~$\le m$.
\end{theorem}

So the $q$-Ehrhart polynomial $\ehr_\poly^\lambda(q,t)$ (naturally) contains more information than $\ehr_\poly(t)$ but still retains some
polynomial structure (as opposed to, say, the multigraded Hilbert series of the cone over~$\poly$). Chapoton's
work has been applied, e.g., to poset enumeration invariants~\cite{kimsong,kimstanton} and graph coloring~\cite{qchromatic}.

\begin{example}\label{ex:triangle}
Let $\Delta$ be the triangle with vertices $(0,0)$, $(1,0)$, and $(0,1)$, and let $\lambda = (1,2)$. Then (as we will compute below in Example~\ref{ex:brioncomp})
\[
  \chap_\Delta^\lambda(q,x) \, = \, \frac{q^3}{q + 1} \, x^2 + \frac{q(2 q + 1)}{q + 1} \, x + 1 \, .
\]
\end{example}

It turns out that, in several special cases, Chapoton polynomials have been computed in the
(sometimes distant) past; Section~\ref{sec:connections} contains some of these instances and
connects them to the underlying polyhedral geometry.
We also remark that~\cite{robinsintpttransform} implies that the set of polynomials $\chap_\poly^\lambda(q,x)$,
where $\lambda$ varies over all valid integral forms, determines the polytope $\poly$, in stark difference to
the classic Ehrhart polynomial.

Our motivating goal is to view Chapoton's Theorem~\ref{thm:chapoton} through the lens of Brion's Theorem~\cite{brion}, which expresses the
integer-point structure of a given polytope via that of its vertex cones.
It turns out that this viewpoint naturally yields various refinements and extensions of
Chapoton's work. We describe their statements next; the accompanying proofs are in Section~\ref{sec:proofs}.

Given a vertex $\bv$ of $\poly$, let $\kegel_\bv$ be the conical hull of all edge directions at
$\bv$, that is,
\[
  \kegel_\bv \, := \sum_{ \bw \text{ adjacent to } \bv } \RR_{ \ge 0 } (\bw - \bv) \, .
\] 
(The affine cone $\bv + \kegel_\bv$ is known as the \Def{vertex cone} of $\poly$ at~$\bv$; we
will use this nomenclature for both $\kegel_\bv$ and $\bv + \kegel_\bv$.)
Given a set $S \subset \RR^d$, we define its \Def{integer-point transform} as
\[
  \sigma_S (\bz) \, := \sum_{ \bm \in S \cap \ZZ^d } \bz^\bm
\]
where $\bz^\bm := z_1^{ m_1 } z_2^{ m_d } \cdots z_d^{ m_d }$. 
When $S$ is a rational cone, $\sigma_S(\bz)$ evaluates to a rational function, and so we may
define, for a given integral form $\lambda$,
\begin{equation}\label{eq:vconerat}
  \rho_\bv^\lambda(q) \, := \, \sigma_{ \kegel_\bv } \left( q^{ \lambda_1 } , \, q^{ \lambda_2 } , \, \dots, \, q^{ \lambda_d } \right) ,
\end{equation}
a rational function in $q$. (Naturally, this is only possible for a form $\lambda$ for which this evaluation
does not create zeros in the denominator. As we will see below, this is equivalent to Chapoton's
genericity condition.)
Our first result gives an explicit formula for Chapoton's $q$-Ehrhart polynomials.

\begin{theorem}\label{thm:firstthm}
If $\poly$ is a lattice polytope and $\lambda$ is an integral form that is generic and positive with respect to $\poly$, then
\[
  \ehr_\poly^\lambda(q,t) \, = \sum_{ \bv \text{ \rm vertex of } \poly } \rho_\bv^\lambda(q) \bigl( (q-1) [t]_q + 1 \bigr)^{ \lambda(\bv) }
\]
where $\rho_\bv^\lambda(q)$ is the rational function defined in~\eqref{eq:vconerat}.
\end{theorem}

Thus 
\[
  \chap_\poly^\lambda(q,x) \, = \sum_{ \bv \text{ \rm vertex of } \poly } \rho_\bv^\lambda(q)
\bigl( (q-1) x + 1 \bigr)^{ \lambda(\bv) } ,
\]
and so Theorem~\ref{thm:firstthm}
immediately recovers the first two statements in Theorem~\ref{thm:chapoton}; we also refine the
third one, as follows. Given a vector $\bm \in \ZZ^d$, we denote by $g(\bm)$ the vector $\bm$ scaled down by the
gcd of the entries of $\bm$. Thus when $\bv$ and $\bw$ are adjacent vertices of $\poly$, then $g(\bw-\bv)$ is
the primitive edge direction from $\bv$ to~$\bw$.

\begin{theorem}\label{thm:poles}
Each pole of $\rho_\bv^\lambda(q)$ (defined in~\eqref{eq:vconerat}) is an $n\th$ root of unity where $n = |\lambda(g(\bw-\bv))|$ for some adjacent vertex~$\bw$.
\end{theorem}

Theorem~\ref{thm:firstthm} can easily be restated in terms of the coefficients of the Chapoton
polynomials:

\begin{corollary}\label{cor:chappolexpresssion}
If $\poly$ is a lattice polytope and $\lambda$ is an integral form that is generic and positive with respect to $\poly$, then
\[
  \chap_\poly^\lambda(q,x) \, = \sum_{ k = 0 }^{ \max \lambda(\bv) } (q-1)^k \sum_{ \bv \text{ \rm vertex
of } \poly } \binom{ \lambda(\bv) } k \, \rho_\bv^\lambda(q) \, x^k
\]
where $\rho_\bv^\lambda(q)$ is the rational function defined in~\eqref{eq:vconerat}.
\end{corollary}

\begin{corollary}
If $\poly$ is a lattice polytope and $\lambda$ is an integral form that is generic and positive with respect to
$\poly$, then the leading coefficient of $\chap_\poly^\lambda(q,x)$ is $(q-1)^{ \lambda(\bv) } \rho_\bv^\lambda(q)$ where $\bv$ is the vertex of $\poly$ that maximizes~$\lambda(\bv)$.
\end{corollary}

In \cite[Section~4.1]{chapoton}, Chapoton suggests to compute $\chap_\poly^\lambda(q, \frac{ 1 }{
1-q })$ (and does so in the case when $\poly$ is an order polytope), as an analogue of computing $\ehr_\poly^\lambda(q,t)$ in the limit as $t \to \infty$.
Theorem~\ref{thm:firstthm} immediately yields this evaluation:

\begin{corollary}\label{cor:ttoinfinity}
If $\poly$ is a lattice polytope and $\lambda$ is an integral form that is generic and positive with respect to $\poly$, then
\[
  \chap_\poly^\lambda \left( q, \frac{ 1 }{ 1-q } \right) 
  \, = \begin{cases}
    \rho_\bzero^\lambda(q) & \text{ if } \bzero \text{ is a vertex of } \poly, \\
    0 & \text{ otherwise. }
  \end{cases}
\]
\end{corollary}

The interpretation $t \to \infty$ is reflected in this formula: in the first case, $\poly$
becomes $\kegel_\bzero$ as $t \to \infty$, whereas in the second case, $\poly$ wanders off into the
horizon.\footnote{
To the Ehrhart experts, we remark that $\chap_\poly^\lambda ( q, x)$ is not translation invariant.
}

Our next result follows from Corollary~\ref{cor:chappolexpresssion}; it was implicitly already proved in~\cite{chapoton}.

\begin{corollary}\label{cor:chappolconstantterm}
If $\poly$ is a lattice polytope and $\lambda$ is an integral form that is generic and positive with respect to $\poly$, then the constant term of $\chap_\poly^\lambda(q,x)$ is $1$.
\end{corollary}

Chapoton~\cite{chapoton} also proved the following combinatorial reciprocity theorem, which lifts
the classic Ehrhart--Macdonald reciprocity theorem \cite{macdonald} into the $q$-world. We will
show that it easily follows from Theorem~\ref{thm:firstthm}.

\begin{theorem}[Chapoton]\label{thm:chapotonrec}
If $\poly$ is a lattice polytope and $\lambda$ is an integral form that is generic and positive with
respect to $\poly$, then
\[ (-1)^{ \dim \poly } \chap_\poly^\lambda \left( \tfrac 1 q, [-t]_{ \frac 1 q} \right) \, = \,
\chap_{\poly^\circ}^\lambda \left( q, [t]_q \right) . \]
Equivalently,
\[
  (-1)^{ \dim \poly } \chap_\poly^\lambda \left( \tfrac 1 q, -qx \right) \, = \, \chap_{\poly^\circ}^\lambda(q,x) \, .
\]
\end{theorem}

Ehrhart counting functions can be considered also for \Def{rational} polytopes, whose vertices
are in $\QQ^d$. Ehrhart~\cite{ehrhartpolynomial} proved that in this case $\ehr_\poly(t)$ is a
\Def{quasipolynomial}, i.e., there exist $p \in \ZZ_{ >0 }$ such that $\ehr_\poly(kp+r)$ is a
polynomial in $k$, for $0 \le r < p$. 
(These polynomials are the \Def{constituents} of the quasipolynomial $\ehr_\poly(t)$.)
Moreover, one can choose $p$ to be the \Def{denominator} of
$\poly$, i.e., the smallest $p$ such that $p \poly$ is a lattice polytope. 
We prove the following rational analogues of Theorems~\ref{thm:chapoton} and~\ref{thm:poles}: 

\begin{theorem}\label{thm:ratlchap}
If $\poly$ is a rational polytope with denominator $p$, and $\lambda$ is an integral form that is generic and positive with respect to $\poly$, then there exist polynomials $\chap_\poly^{\lambda, r}(q,x) \in \QQ(q)[x]$ such that 
\[ \chap_\poly^{\lambda, r} \left(q, [k]_q \right) \, = \, \ehr_\poly^\lambda(q,kp+r) \] for all
integers $k \ge 0$ and all $0 \le r < p$.
The degree of $\chap_\poly^{\lambda, r}(q,x)$ is $\max \{ \lambda(p \bv) : \, \bv \text{ \rm vertex of } \poly \}$.
Each pole of a coefficient of $\chap_\poly^{\lambda, r}(q,x)$ is an $n\th$ root of unity where $n
= |\lambda(g(p(\bw-\bv)))|$ for some adjacent vertices $\bv$ and~$\bw$.
\end{theorem}

Corollaries~\ref{cor:ttoinfinity} and~\ref{cor:chappolconstantterm} take on the following forms when
$\poly$ is rational.

\begin{corollary}
If $\poly$ is a rational polytope with denominator $p$, and $\lambda$ is an integral form which
is generic and positive with respect to $\poly$, then for any $0 \le r < p$,
\[
  \chap_\poly^{\lambda, r} \left( q, \frac{ 1 }{ 1-q } \right) 
  \, = \begin{cases}
    \rho_\bzero^\lambda(q) & \text{ if } \bzero \text{ is a vertex of } \poly, \\
    0 & \text{ otherwise. }
  \end{cases}
\]
\end{corollary}

\begin{corollary}\label{cor:quasiconstantterm}
If $\poly$ is a rational polytope and $\lambda$ is an integral form that is generic and positive with
respect to $\poly$, then the constant term of $\chap_\poly^{\lambda, 0} (q,x)$ is $1$.
\end{corollary}

Finally, we offer the following rational version of Theorem~\ref{thm:chapotonrec}.

\begin{theorem}\label{thm:rationalchapotonrec}
If $\poly$ is a rational polytope with denominator $p$, and $\lambda$ is an integral form which
is generic and positive with respect to $\poly$, then for any $0 \le r < p$ and $k > 0$,
\[
  (-1)^{ \dim \poly } \chap_\poly^{\lambda, r} \left( \tfrac 1 q, [-k]_{ \frac 1 q} \right) \, = \,
\ehr_{\poly^\circ}^\lambda(q, kp-r) \, .
\]
\end{theorem}


\section{Connections}\label{sec:connections}

We now develop both a toolbox of examples of Chapoton polynomials and give a glimpse of history
where they had previously appeared (naturally with different nomenclature).
For these examples, as well as the proofs in Section~\ref{sec:proofs}, we make the simple but crucial observation that, by definition of~$[t]_q$, 
\begin{equation}\label{eq:qpowers}
  q^{ kt } \, = \, \bigl( (q-1) [t]_q + 1 \bigr)^k ,
\end{equation}
and so $q^{ kt }$ can be expressed as a degree-$k$ polynomial in $[t]_q$ with coefficients in $\QQ[q]$.

\subsection{The unit cube and Euler--Mahonian statistics}

Let $\Box = [0,1]^d$ and $\lambda = (1, 1, \dots, 1)$. Then
\begin{equation}\label{eq:unitcubechap}
  \ehr_\Box^\lambda(q,t) \, = \, [t+1]_q^d
  \qquad \text{ and so } \qquad
  \chap_\Box^\lambda(q,x) \, = \, (1 + q x)^d .
\end{equation}
The quantity $[t+1]_q^d$ famously appears in the \emph{Carlitz identity}
~\cite{carlitzeulerian} (though with some effort one can derive it from the works of MacMahon \cite[Volume 2, Chapter IV, \S462]{macmahon})
\begin{equation}\label{eq:carlitz}
  \sum_{ t \ge 0 } [t+1]_q^n \, x^t
  \, = \, \frac{ \sum_{ \pi \in S_n } x^{ \des (\pi) } q^{ \maj (\pi) } }{ \prod_{ j=0 }^n \left( 1 - x q^j \right) } \, .
\end{equation}
The numerator on the right is known as an \emph{Euler--Mahonian polynomial} due to the relation
with Euler's work on $\des(\pi) := |\Des(\pi)|$ and MacMahon's original introduction of the major index
$\maj(\pi) := \sum_{ j \in \Des(\pi) } j$
for a permutation
$\pi \in S_n$, where $\Des(\pi) := \{ j : \, \pi(j) > \pi(j+1) \}$.
We remark that~\cite{tielker} introduced $q$-weighted Ehrhart polynomials beyond those defined by Chapoton to
derive various extensions of~\eqref{eq:carlitz}, in particular involving signed permutations.

The vertices of \(\Box\) are of the form \(\be_{I} := \sum_{j \in I}\be_j\), where \(I \subseteq [d] := \{ 1, 2,
\dots, d \} \) and
\(\be_j\) is the \(j\)th standard unit vector. Each vertex lies in \(d\) edges; if \(j \in I\)
then \(-\be_j\) is an edge direction at $\be_I$, otherwise \(\be_j\) is an edge direction. Since each vertex cone of \(\Box\) is unimodular, 
\[
  \rho_{ \be_I }^\lambda(q)
  \, = \, \frac{ 1 }{ (1-q)^{ d-|I| } ( 1 - \frac 1 q )^{ |I| } } 
  \, = \, \frac{ (-q)^{ |I| } }{ (1-q)^d } 
\]
and so Theorem~\ref{thm:firstthm} gives
\[
  \chap_\Box^\lambda(q,x) = \sum_{j = 0}^{d} \left( d \atop j \right) \frac{(-q)^j}{\left(1 - q  \right)^{d}}
  \bigl( (q-1) x + 1 \bigr)^{j}.
\]
One can apply the binomial theorem to recover the arguably simpler expression in~\eqref{eq:unitcubechap}.

For general $\lambda > 0$,
\[
  \ehr_\Box^\lambda(q,t)
  \, = \, \prod_{ j=1 }^d [t+1]_{ q^{ \lambda_j } }
  \, = \, \prod_{ j=1 }^d \frac{ q^{ \lambda_j(t+1) } - 1 }{ q^{ \lambda_j } - 1 } 
  \, = \, \prod_{ j=1 }^d \frac{ q^{ \lambda_j } \left( (q-1) [t]_q + 1 \right)^{ \lambda_j } - 1
}{ q^{ \lambda_j } - 1 }  \, ,
\]
and this function appears in a multivariate refinement of~\eqref{eq:carlitz}~\cite{eulermahonian}.
We can compute its incarnation via Theorem~\ref{thm:firstthm} as above: with
\(\lambda_{I} := \sum_{j \in I}\lambda_j\), we have 
\[
  \chap_\Box^\lambda(q,x) \, = \sum_{I \subseteq [d]} \frac{(-1)^{|I|}q^{\lambda_{I}}}{\prod_{j = 1}^{d}\left(1 - q^{\lambda_j} \right)} \bigl( (q-1) x + 1 \bigr)^{ \lambda_{I} }.
\]

\subsection{The standard simplex and Sylvester waves}

Let $\Delta = \left\{ \bx \in \RR_{ \ge 0 }^d : \, x_1 + x_2 + \dots + x_d = 1 \right\}$, the
convex hull of the $d$ standard unit vectors, and assume that the coefficients of $\lambda$ are
distinct.
Then
\[
  \ehr_\Delta^\lambda(q,t)
  \, = \sum_{ \bm \in t \Delta } q^{ \lambda_1 m_1 + \lambda_2 m_2 + \dots + \lambda_d m_d }
\]
is the generating function of the number of partitions with exactly $t$ parts in the set $\Lambda
:= \{ \lambda_1, \lambda_2, \dots, \lambda_d \}$.
This is a refinement of sorts of the \emph{restricted partition function} $p_\Lambda(n)$, the
number of partitions of $n$ with parts in $\Lambda$. This function is classically decomposed into
\emph{Sylvester waves} and is intimately connected to the \emph{Frobenius coin-exchange problem}
(see, e.g., \cite[Chapter~1]{ccd}).

The vertex cone at $\be_j$ is again unimodular, with edge directions $\be_k - \be_j$, so that
\[
  \rho_{ \be_j }^\lambda(q)
  \, = \, \frac{ 1 }{ \prod_{ k \ne j } \left( 1 - q^{ \lambda_k - \lambda_j } \right) } \, ,
\]
and Theorem~\ref{thm:firstthm} yields 
\[
  \chap_\Delta^\lambda(q,x) \, = \, \sum_{ j=1 }^{ d } \frac{ 1 }{ \prod_{ k \ne j } \left( 1 - q^{ \lambda_k - \lambda_j } \right) } \bigl( (q-1) x + 1 \bigr)^{ \lambda_{j} }.
\]

\subsection{Unimodular simplices and integer partitions}

Let $\Delta = \left\{ \bx \in \RR^d : \, 0 \le x_1 \le x_2 \le \dots \le x_d \le 1 \right\}$,
another unimodular $d$-simplex, and let $\lambda = (1, 1, \dots, 1)$.
Then
\[
  \ehr_\Delta^\lambda(q,t)
  \, = \sum_{ \bm \in t \Delta } q^{ m_1 + m_2 + \dots + m_d }
  \, = \, \left[ {t+d} \atop d \right]_q \, , 
\]
the (very classical) generating function of the number of partitions with at most $d$ parts, each of which $\le t$ (see, e.g.,~\cite{andrewstheoryofpartitions}). Thus
\[
  \chap_\Delta^\lambda(q,x)
  \, = \, \frac{ 1 }{ [d]_q! } \prod_{ j=1 }^d \left( q^j x + [j]_q \right) .
\]

Along with the origin, the vertices of \(\Delta\) are $\be_d, \be_{d - 1} + \be_{d}, \dots, \be_1
+ \dots + \be_{d - 1} + \be_d$. As \(\lambda(\be_{d - k + 1} + \dots + \be_d) = k\), we have
\[
  \rho_{ \be_j }^\lambda(q)
  \, = \, \frac{1}{\prod_{k \neq j}\left( 1 - q^{k - j}
\right)} \, ,
\]
and Theorem~\ref{thm:firstthm} gives 
\[
  \chap_\Delta^\lambda(q,x) \, = \, \sum_{j = 0}^{d}\frac{1}{\prod_{k \neq j}\left( 1 - q^{k - j}
\right)}\bigl((q - 1)x + 1 \bigr)^{j}. 
\]

\subsection{Lecture hall simplices/partitions}
Our last family of examples is given by the \Def{lecture hall simplex}
\[
  \Delta_n \, := \, \left\{ \bx \in [0,1]^n : \, x_1 \le \frac{ x_2 }{ 2 } \le \frac{ x_3 }{ 3 } \le \dots \le \frac{ x_n }{ n } \right\} ,
\]
first studied in~\cite{corteelleesavage} in
conjunction with the famous \Def{lecture hall partitions}, which we may view as points in
the cone
\[
  \kegel_n \, := \, \left\{ \bx \in \RR_{ \ge 0 }^n : \, x_1 \le \frac{ x_2 }{ 2 } \le \frac{ x_3 }{ 3 } \le \dots \le \frac{ x_n }{ n } \right\}
\]
---a lecture hall partition of $m$ is an integer point in $\kegel_n$ (where $n$ equals the number
of parts of the partition) whose coordinates sum to~$m$.
Lecture hall partition first saw the light of day in~\cite{BME1} and have been the subject of
active research; \cite{savagesurvey} gives a survey.

The lecture hall cone $\kegel_n$ is the vertex cone of $\Delta_n$ at the origin.
With $\lambda = (1, 1, \dots, 1)$, the rational function $\rho_\bzero^\lambda(q)$ (in the
language of \eqref{eq:vconerat}) is precisely the rational generating function of all lecture
hall partitions with $\le n$ parts.

The lecture hall simplex $\Delta_n$ is rational (with denominator $n$), and so the recursion
\begin{equation}\label{lhrecursion}
  \ehr_{ \Delta_n }^\lambda(q,jn+i) \, = \, \ehr_{ \Delta_n }^\lambda(q,jn+i-1) + q^{ jn+i }
\ehr_{ \Delta_{n-1} }^\lambda(q,j(n-1)+i-1) \, ,
\end{equation}
for any $j \ge 0$ and $1 \le i \le n$,
which follows from first principles~\cite{corteelleesavage}, might historically be the earliest
occurrence of a $q$-Ehrhart \emph{quasi}polynomial.
In terms of the Chapoton constituents, the recursion can be translated as follows.
Define $\chap_{ n, i } (x) \in \QQ(q)[x]$ via
\[
  \chap_{ n,i } \left( [j]_q \right) \, = \, \ehr_{ \Delta_n }^\lambda(q,jn+i)
\]
for $j \ge 0$ and $1 \le i \le n$ (noting that \(\chap_{n, i}(x)\) here is \(\chap_{\Delta_n}^{\lambda, i}(x)\) in Theorem~\ref{thm:ratlchap}), and define $\chap_{ n, 0 } (x) \in \QQ(q)[x]$ via
\begin{equation}\label{eq:lhchapsymm}
  \chap_{ n, n } (x) \, = \, \chap_{ n, 0 } (1+qx) \, ;
\end{equation}
this last definition is motivated by
\[
  \chap_{ n,n } \left( [j]_q \right)
  \, = \, \ehr_{ \Delta_n }^\lambda(q,jn+n)
  \, = \, \ehr_{ \Delta_n }^\lambda(q,(j+1)n)
  \, = \, \chap_{ n,0 } \left( [j+1]_q \right) .
\]
Thus we can start with
\begin{equation}\label{eq:lhchapstart}
  \chap_{ 1, 0 } (x) \, = \, 1+qx
  \qquad \text{ and } \qquad
  \chap_{ 1, 1 } (x) \, = \, 1+q+q^2x \, ,
\end{equation}
and \eqref{lhrecursion} yields
\[
  \chap_{ n, i } (x) \, = \, \chap_{ n, i-1 } (x) + q^i \bigl( (q-1)x + 1 \bigr)^n \chap_{ n-1, i-1 } (x)
\]
for $n \ge 2$ and $1 \le i \le n$.
Iterating this relation gives
\begin{equation}\label{eq:lhchaprec}
  \chap_{ n, i } (x) \, = \, \chap_{ n, 0 } (x) + \bigl( (q-1)x + 1 \bigr)^n \sum_{ j=1 }^{ i } q^j \, \chap_{ n-1, j-1 } (x)
\end{equation}
and so, in particular, via~\eqref{eq:lhchapsymm},
\begin{equation}\label{eq:qdiffequ}
  \chap_{ n, 0 } (1+qx) - \chap_{ n, 0 } (x) \, = \, \bigl( (q-1)x + 1 \bigr)^n \sum_{ j=0 }^{ n-1 }
q^{j+1} \, \chap_{ n-1, j } (x) \, .
\end{equation}
Connaisseurs of $q$-calculus will realize~\eqref{eq:qdiffequ} as a $q$-difference equation, and so
\eqref{eq:lhchapstart}, \eqref{eq:lhchaprec}, and~\eqref{eq:qdiffequ}
completely determine the family of polynomials $\chap_{ n, i } (x)$ up to a constant, but the latter is
determined by
\begin{equation}\label{}
  \chap_{ n, 0 } (0) \, = \, 1 \, , 
\end{equation}
which holds by Corollary~\ref{cor:quasiconstantterm}.
The determination of the actual polynomials seems subtle;
we note that they were not computed in closed form in~\cite{corteelleesavage}.

We conclude this section by remarking that, by Corollary~\ref{cor:ttoinfinity}, $\chap_{ n,0 } ( \frac{ 1 }{ 1-q })$ gives yet another incarnation of the rational generating function of all lecture hall partitions with $\le n$ parts.


\section{Proofs}\label{sec:proofs}

Our main ingredient is the following~\cite{brion}.

\begin{theorem}[Brion]\label{thm:brion}
If $\poly$ is a rational polytope, we have the following identity of rational functions:
\[
  \sigma_\poly(\bz) \, = \sum_{ \bv \text{ \rm vertex of } \poly } \sigma_{ \bv + \kegel_\bv } (\bz) \, .
\]
\end{theorem}

We are ready to prove our refinement of Chapoton's Theorem.

\begin{proof}[Proof of Theorem~\ref{thm:firstthm}]
By Brion's Theorem~\ref{thm:brion},
\begin{align}
  \ehr_\poly^\lambda(q,t)
  \, &= \sum_{ \bm \in t \poly } q^{ \lambda(\bm) } 
  \, = \, \sigma_{ t \poly } \left( q^{ \lambda_1 } , q^{ \lambda_2 } , \dots, q^{ \lambda_d } \right)
  \, = \sum_{ \bv \text{ \rm vertex of } \poly } \sigma_{ t \bv + \kegel_\bv } \left( q^{ \lambda_1 } , q^{ \lambda_2 } , \dots, q^{ \lambda_d } \right) \nonumber \\
  \, &= \sum_{ \bv \text{ \rm vertex of } \poly } q^{ t \lambda(\bv) } \, \sigma_{ \kegel_\bv } \left( q^{ \lambda_1 } , q^{ \lambda_2 } , \dots, q^{ \lambda_d } \right) . \label{eq:vertexconespecialization}
\end{align}
Now use~\eqref{eq:qpowers}.
\end{proof}

\begin{example}\label{ex:brioncomp}
We will now compute the Chapoton polynomial of Example~\ref{ex:triangle} via Theorem~\ref{thm:firstthm}.
Since each vertex cone of $\Delta$ is unimodular,
\begin{align}
  \sigma_{ \kegel_{ (0,0) } } \left(q,q^2 \right) \, &= \, \frac{ 1 }{ \left(1-q \right) \left(1-q^2 \right) } \nonumber \\
  \sigma_{ \kegel_{ (1,0) } } \left(q,q^2 \right) \, &= \, \frac{ 1 }{ \left( 1 - \frac 1 q \right) \left(1-q \right) } 
     \, = \, - \, \frac{ q }{ (1-q)^2 } \label{eq:samplegenfcts} \\
  \sigma_{ \kegel_{ (0,1) } } \left( q,q^2 \right) \, &= \, \frac{ 1 }{ \left( 1 - \frac 1 {q^2} \right) \left( 1 - \frac 1 q \right) } 
     \, = \, \frac{ q^3 }{ \left(1-q \right) \left(1-q^2 \right) } \, , \nonumber
\end{align}
and so Theorem~\ref{thm:firstthm} yields
\begin{align*}
  \ehr_\Delta^\lambda(q,t)
  \, &= \, \frac{ 1 }{ \left(1-q \right) \left(1-q^2 \right) } - \frac{ q }{ (1-q)^2 } \bigl( (q-1) [t]_q + 1
\bigr) + \frac{ q^3 }{ \left(1-q \right) \left(1-q^2 \right) } \bigl( (q-1) [t]_q + 1 \bigr)^2 \\
  \, &= \, \frac{ q^3 }{ 1+q } \, [t]_q^2 + \frac{ q + 2q^2 }{ 1+q } \, [t]_q + 1 \, .
\end{align*}
In this example, we can also see the following proof in action when looking at the denominators
of the vertex cone generating functions in~\eqref{eq:samplegenfcts}.
\end{example}

\begin{proof}[Proof of Theorem~\ref{thm:poles}]
It follows from first principles (see, e.g., \cite[Section~3.3]{ccd}) that the rational-function
form of $\sigma_\kegel(\bz)$, where $\kegel$ is a \emph{simplicial} rational cone (i.e., $\kegel$
has dimension many generators), has denominator 
$\prod_{ \bg } (1 - \bz^\bg)$ where the product is taken over all generators of $\kegel$.
Because every cone can be triangulated without introducing new generators, the same is true for a
general rational cone~$\kegel$.

The generators of the vertex cone $\kegel_\bv$ of a given polytope are $g(\bw - \bv)$ for all
vertices $\bw$ that are adjacent to~$\bv$.
Thus specializing the integer-point transform of $\kegel_\bv$ according to
\eqref{eq:vertexconespecialization} yields the denominator of
$
  \sigma_{ \kegel_\bv } \left( q^{ \lambda_1 } , q^{ \lambda_2 } , \dots, q^{ \lambda_d } \right)
$
as
\[
  \prod_{ \bw \text{ adjacent to } \bv } \left( 1 - q^{ \lambda(g(\bw-\bv)) } \right) . \qedhere
\]
\end{proof}

\begin{proof}[Proof of Corollary~\ref{cor:chappolconstantterm}]
Translate the given lattice polytope $\poly \subset \RR^d$ to the (rational) polytope $\qoly$ so that the
barycenter of $\qoly$ is the origin; note that two corresponding vertices of $\poly$ and $\qoly$
come with identical cones $\kegel_\bv$.
Because $\qoly$ is a rational, we can further shrink it such that for each of its vertices,
\[
  \left( \bv + \kegel_\bv \right) \cap \ZZ^d \, = \, \kegel_\bv \cap \ZZ^d .
\]
Note that this implies $\qoly \cap \ZZ^d = \{ \bzero \}$, and thus Brion's Theorem~\ref{thm:brion} applied to
$\qoly$ says
\[
  1
  \, = \sum_{ \bv \text{ vertex of } \qoly } \sigma_{ \bv + \kegel_\bv } (\bz) 
  \, = \sum_{ \bv \text{ vertex of } \qoly } \sigma_{ \kegel_\bv } (\bz) 
  \, = \sum_{ \bv \text{ vertex of } \poly } \sigma_{ \kegel_\bv } (\bz) \, .
\]
On the other hand, by Corollary~\ref{cor:chappolexpresssion}, the constant term of
$\chap_\poly^\lambda(q,x)$ equals
\[
  \sum_{ \bv \text{ vertex of } \poly } \rho_\bv^\lambda(q)
  \, = \sum_{ \bv \text{ vertex of } \poly } \sigma_{ \kegel_\bv } \left( q^{ \lambda_1 } , \ q^{ \lambda_2 } , \ \dots, \ q^{ \lambda_d } \right)
  \, = \, 1 \, . \qedhere
\]
\end{proof}

For our proof of Theorem~\ref{thm:chapotonrec}, we will need to following reciprocity
theorem~\cite{stanleyreciprocity}. We denote the relative interior of $\kegel$ as $\kegel^\circ$
and write $\frac 1 \bz := (\frac{ 1 }{ z_1 }, \frac{ 1 }{ z_2 }, \dots, \frac{ 1 }{ z_d })$.

\begin{theorem}[Stanley]\label{thm:stanleyrec}
Let $\kegel \subseteq \RR^d$ be a rational cone. Then, as rational functions,
\[
  (-1)^{ \dim \kegel } \sigma_\kegel \left( \tfrac 1 \bz \right) \, = \, \sigma_{ \kegel^\circ }
(\bz) \, .
\]
\end{theorem}

\begin{proof}[Proof of Theorem~\ref{thm:chapotonrec}]
The variant of Brion's Theorem~\ref{thm:brion} for $\poly^\circ$ (see, e.g.,
\cite[Exercise~11.9]{ccd}) says
\[
  \sigma_{ \poly^\circ } (\bz) \, = \sum_{ \bv \text{ \rm vertex of } \poly } \sigma_{ \bv +
\kegel_\bv^\circ } (\bz) \, .
\]
Thus
\[
  \ehr_{\poly^\circ}^\lambda \left( \tfrac 1 q, t \right)
  \, = \sum_{ \bv \text{ \rm vertex of } \poly } q^{ -t \lambda(\bv) } \, \sigma_{
\kegel_\bv^\circ } \left( q^{ -\lambda_1 } , q^{ -\lambda_2 } , \dots, q^{ -\lambda_d } \right) 
\]
and we can apply Theorem~\ref{thm:stanleyrec} to the above summands to deduce
\[
  \ehr_{\poly^\circ}^\lambda \left( \tfrac 1 q, t \right)
  \, = \, (-1)^{ \dim \poly }  \sum_{ \bv \text{ \rm vertex of } \poly } q^{ -t \lambda(\bv) } \,
\sigma_{ \kegel_\bv } \left( q^{ \lambda_1 } , q^{ \lambda_2 } , \dots, q^{ \lambda_d } \right) 
  \, \stackrel{ \eqref{eq:vertexconespecialization} }{ = } \, (-1)^{ \dim \poly }
\chap_\poly^\lambda (q, [-t]_q) \, .
\]

The second claim follows with
\[
  [-t]_{ \frac 1 q }
  \, = \, \frac{ 1-q^t }{ 1 - \frac 1 q } 
  \, = \, q \, \frac{ 1-q^t }{ q - 1 }
  \, = \, - q \, [t]_q \, . \qedhere
\] 
\end{proof}


\begin{proof}[Proof of Theorem~\ref{thm:ratlchap}]
By Brion's Theorem~\ref{thm:brion},
\begin{align*}
  &\ehr_\poly^\lambda(q,kp+r)
  \, = \sum_{ \bm \in (kp+r) \poly } q^{ \lambda(\bm) } 
  \, = \, \sigma_{ (kp+r) \poly } \left( q^{ \lambda_1 } , q^{ \lambda_2 } , \dots, q^{ \lambda_d } \right) \\
  &\qquad = \sum_{ \bv \text{ \rm vertex of } \poly } \sigma_{ (kp+r) \bv + \kegel_\bv } \left( q^{ \lambda_1 } , q^{ \lambda_2 } , \dots, q^{ \lambda_d } \right) 
  \, = \sum_{ \bv \text{ \rm vertex of } \poly } q^{ k \, \lambda(p \bv) } \, \sigma_{ r \bv + \kegel_\bv } \left( q^{ \lambda_1 } , q^{ \lambda_2 } , \dots, q^{ \lambda_d } \right) .
\end{align*}
With~\eqref{eq:qpowers}, this yields
\[
  \ehr_\poly^\lambda(q,kp+r)
  \, = \sum_{ \bv \text{ \rm vertex of } \poly } \sigma_{ r \bv + \kegel_\bv } \left( q^{ \lambda_1 } , q^{ \lambda_2 } , \dots, q^{ \lambda_d } \right) \bigl( (q-1) [k]_q + 1 \bigr)^{ \lambda(p \bv) } ,
\]
which can be viewed as the rational version of Theorem~\ref{thm:firstthm}.
So the polynomial we are after is
\[
  \chap_\poly^{\lambda, r} (q,x) \, = \, \sum_{ \bv \text{ \rm vertex of } \poly } \sigma_{ r \bv
+ \kegel_\bv } \left( q^{ \lambda_1 } , q^{ \lambda_2 } , \dots, q^{ \lambda_d } \right) \bigl(
(q-1) x + 1 \bigr)^{ \lambda(p \bv) } ,
\]
which has degree $\max \{ \lambda(p \bv) : \, \bv \text{ \rm vertex of } \poly \}$.

Our proof of Theorem~\ref{thm:poles} extends almost verbatim to the rational case, as the
generators of $r \bv + \kegel_\bv$ are still $\bw - \bv$ for all vertices $\bw$ that are adjacent
to $\bv$; the only difference is that their integral version is now $g(p(\bw-\bv))$.
\end{proof}

\begin{proof}[Proof of Theorem~\ref{thm:rationalchapotonrec}]
We start as in our proof of Theorem~\ref{thm:chapotonrec}:
by Brion's Theorem~\ref{thm:brion} for $\poly^\circ$,
\[
  \sigma_{ \poly^\circ } (\bz) \, = \sum_{ \bv \text{ \rm vertex of } \poly } \sigma_{ \bv +
\kegel_\bv^\circ } (\bz) \, ,
\]
and so
\begin{align*}
  \ehr_{\poly^\circ}^\lambda \left( \tfrac 1 q, \ kp-r \right)
  \, &= \sum_{ \bv \text{ \rm vertex of } \poly } \sigma_{ (kp-r)\bv + \kegel_\bv^\circ } \left(
q^{ -\lambda_1 } , q^{ -\lambda_2 } , \dots, q^{ -\lambda_d } \right) \\
     &= \sum_{ \bv \text{ \rm vertex of } \poly } q^{ -k \lambda(p \bv) } \sigma_{ -r \bv + \kegel_\bv^\circ } \left( q^{ -\lambda_1 } , q^{ -\lambda_2 } , \dots, q^{ -\lambda_d } \right) .
\end{align*}
We now apply a variant of Stanley's Theorem~\ref{thm:stanleyrec} (see, e.g.,
\cite[Section~4.2]{ccd}), namely for a cone $\kegel$ with apex at the origin and any $\bv \in
\RR^d$, we have the following equation of rational functions:
\[
  \sigma_{ \bv + \kegel } (\bz) \, = \, (-1)^{ \dim \kegel } \sigma_{ -\bv + \kegel^\circ }
\left( \tfrac 1 \bz \right) .
\]
Thus
\begin{align*}
  \ehr_{\poly^\circ}^\lambda \left( \tfrac 1 q, \ kp-r \right)
  \, &= (-1)^{ \dim \poly } \sum_{ \bv \text{ \rm vertex of } \poly } q^{ -k \lambda(p \bv) }
\sigma_{ r \bv + \kegel_\bv } \left( q^{ \lambda_1 } , q^{ \lambda_2 } , \dots, q^{ \lambda_d }
\right) \\
     &= \, (-1)^{ \dim \poly } \chap_\poly^{\lambda, r} \left( q, [-k]_q \right) .
\qedhere
\end{align*}
\end{proof}


\bibliographystyle{amsplain}
\bibliography{bib}

\def\cprime{$'$} \def\cprime{$'$}
\providecommand{\bysame}{\leavevmode\hbox to3em{\hrulefill}\thinspace}
\providecommand{\MR}{\relax\ifhmode\unskip\space\fi MR }
\providecommand{\MRhref}[2]{%
  \href{http://www.ams.org/mathscinet-getitem?mr=#1}{#2}
}
\providecommand{\href}[2]{#2}
\begin{thebibliography}{10}

\bibitem{andrewstheoryofpartitions}
George~E. Andrews, \emph{The {T}heory of {P}artitions}, Cambridge Mathematical
  Library, Cambridge University Press, Cambridge, 1998, reprint of the 1976
  original.

\bibitem{qchromatic}
Esme Bajo, Matthias Beck, and Andr\'es~R. Vindas-Mel\'endez,
  \emph{$q$-chromatic polynomials}, 2024, Preprint ({\tt arXiv:2403.19573}).

\bibitem{eulermahonian}
Matthias Beck and Benjamin Braun, \emph{Euler--{M}ahonian statistics via
  polyhedral geometry}, Adv. Math. \textbf{244} (2013), 925--954, {\tt
  arXiv:1109.3353}.

\bibitem{ccd}
Matthias Beck and Sinai Robins, \emph{Computing the {C}ontinuous {D}iscretely:
  Integer-point {E}numeration in {P}olyhedra}, second ed., Undergraduate Texts
  in Mathematics, Springer, New York, 2015.

\bibitem{BME1}
Mireille Bousquet-M{\'e}lou and Kimmo Eriksson, \emph{Lecture hall partitions},
  Ramanujan J. \textbf{1} (1997), no.~1, 101--111.

\bibitem{brion}
Michel Brion, \emph{Points entiers dans les poly\`edres convexes}, Ann. Sci.
  \'Ecole Norm. Sup. (4) \textbf{21} (1988), no.~4, 653--663.

\bibitem{carlitzeulerian}
Leonard Carlitz, \emph{A combinatorial property of {$q$}-{E}ulerian numbers},
  Amer. Math. Monthly \textbf{82} (1975), 51--54.

\bibitem{chapoton}
Fr\'{e}d\'{e}ric Chapoton, \emph{{$q$}-analogues of {E}hrhart polynomials},
  Proc. Edinb. Math. Soc. (2) \textbf{59} (2016), no.~2, 339--358.

\bibitem{corteelleesavage}
Sylvie Corteel, Sunyoung Lee, and Carla~D. Savage, \emph{Enumeration of
  sequences constrained by the ratio of consecutive parts}, S{\'e}min. Lothar.
  Comb. \textbf{54A} (2005), Paper B54Aa (12 pages).

\bibitem{ehrhartpolynomial}
Eug{\`e}ne Ehrhart, \emph{Sur les poly\`edres rationnels homoth\'etiques \`a
  {$n$}\ dimensions}, C. R. Acad. Sci. Paris \textbf{254} (1962), 616--618.

\bibitem{kimsong}
Jang~Soo Kim and U-Keun Song, \emph{Proof of {Chapoton}'s conjecture on
  {Newton} polygons of {{\(q\)}}-{Ehrhart} polynomials}, Electron. J. Comb.
  \textbf{25} (2018), no.~2, research paper P2.51, 14 pages.

\bibitem{kimstanton}
Jang~Soo Kim and Dennis Stanton, \emph{On {{\(q\)}}-integrals over order
  polytopes}, Adv. Math. \textbf{308} (2017), 1269--1317.

\bibitem{macdonald}
Ian~G. Macdonald, \emph{Polynomials associated with finite cell-complexes}, J.
  London Math. Soc. (2) \textbf{4} (1971), 181--192.

\bibitem{macmahon}
Percy~A. MacMahon, \emph{Combinatory {A}nalysis}, Chelsea Publishing Co., New
  York, 1960, reprint of the 1915 original.

\bibitem{robinsintpttransform}
Sinai Robins, \emph{The integer point transform as a complete invariant},
  Commun. Math. \textbf{31} (2023), no.~2, 157--172.

\bibitem{savagesurvey}
Carla~D. Savage, \emph{The mathematics of lecture hall partitions}, J. Comb.
  Theory, Ser. A \textbf{144} (2016), 443--475.

\bibitem{stanleyreciprocity}
Richard~P. Stanley, \emph{Combinatorial reciprocity theorems}, Advances in
  Math. \textbf{14} (1974), 194--253.

\bibitem{tielker}
Elena Tielker, \emph{Weighted {Ehrhart} series and a type-{{\(\mathsf{B}\)}}
  analogue of a formula of {MacMahon}}, S{\'e}min. Lothar. Comb. \textbf{86B}
  (2022), 12 pages, Id/No 25.

\end{thebibliography}

\setlength{\parskip}{0cm} 

\end{document}